\documentclass[11pt]{amsart}
\usepackage{amsmath}
\usepackage{amsthm}
\usepackage{amssymb}
\usepackage{graphicx}
\usepackage{subcaption}

\numberwithin{equation}{section}
\newtheorem{thm}{Theorem}[section]

\newtheorem{lemma}[thm]{Lemma}
\newtheorem{prop}[thm]{Proposition}

\newtheorem{definition}[thm]{Definition}
\newtheorem{remark}[thm]{Remark}
\newtheorem{exam}[thm]{Example}

\begin{document}
\setlength{\parindent}{0em}
\title{On a diffusion equation with rupture}

\author{Yoshikazu Giga}
\address{Graduate School of Mathematical Sciences\\
University of Tokyo\\  3-8-1 Komaba Meguro-ku\\   Tokyo 153-8914 \\ Japan}
\email{labgiga@ms.u-tokyo.ac.jp}

\author{Yuki Ueda}
\address{Graduate School of Mathematical Sciences\\
University of Tokyo\\  3-8-1 Komaba Meguro-ku\\   Tokyo 153-8914 \\ Japan}
\email{yukiueda@ms.u-tokyo.ac.jp}

\begin{abstract}
We propose a model to describe an evolution of a bubble cluster with rupture.
 In a special case, the equation is reduced to a single parabolic equation with evaporation for the thickness of a liquid layer covering bubbles.
 We postulate that a bubble collapses if this liquid layer becomes thin.
 We call this collapse a rupture.
 We prove for our model that there is a periodic-in-time solution if the place of rupture occurs only in the largest  bubble.
 Numerical tests indicate that there may not exist a periodic solution if such an assumption is violated.
\end{abstract}

\renewcommand{\thefootnote}{\fnsymbol{footnote}}
\renewcommand{\thefootnote}{\arabic{footnote}}

\date{}
\maketitle


\bibliographystyle{plain}

\section{Introduction} \label{S1} 
\subsection{Setting of the problem} \label{S1S1} 

We consider a cluster of bubbles which occupies a domain in a plane $\mathbf{R}^2$ whose boundary is the graph of a function $h$.
 We postulate that a bubble touches the graph $y=h(x)$ (from below) on a fixed horizontal region and that the difference of slopes of adjacent bubbles is fixed.
 We further assume that $h$ is periodic with period $\omega$ to simplify the problem.
 To fix the idea, let $U_k$ denote the $k$-th bubble on $(a_k,a_{k+1})$ whose boundary contains $\left\{y=h(x) \bigm| a_k<x< a_{k+1}\right\}$, where $\{a_k\}_{k=1}^K$ is a division of $[0,\omega)$, i.e., $0\leq a_1<a_2<\cdots<\omega$;
 see Figure \ref{Fb}.
 The index is taken modulo $K$ so that $a_{1+K}=a_1$.
\begin{figure}[htb]
\centering
\includegraphics[width=7.5cm]{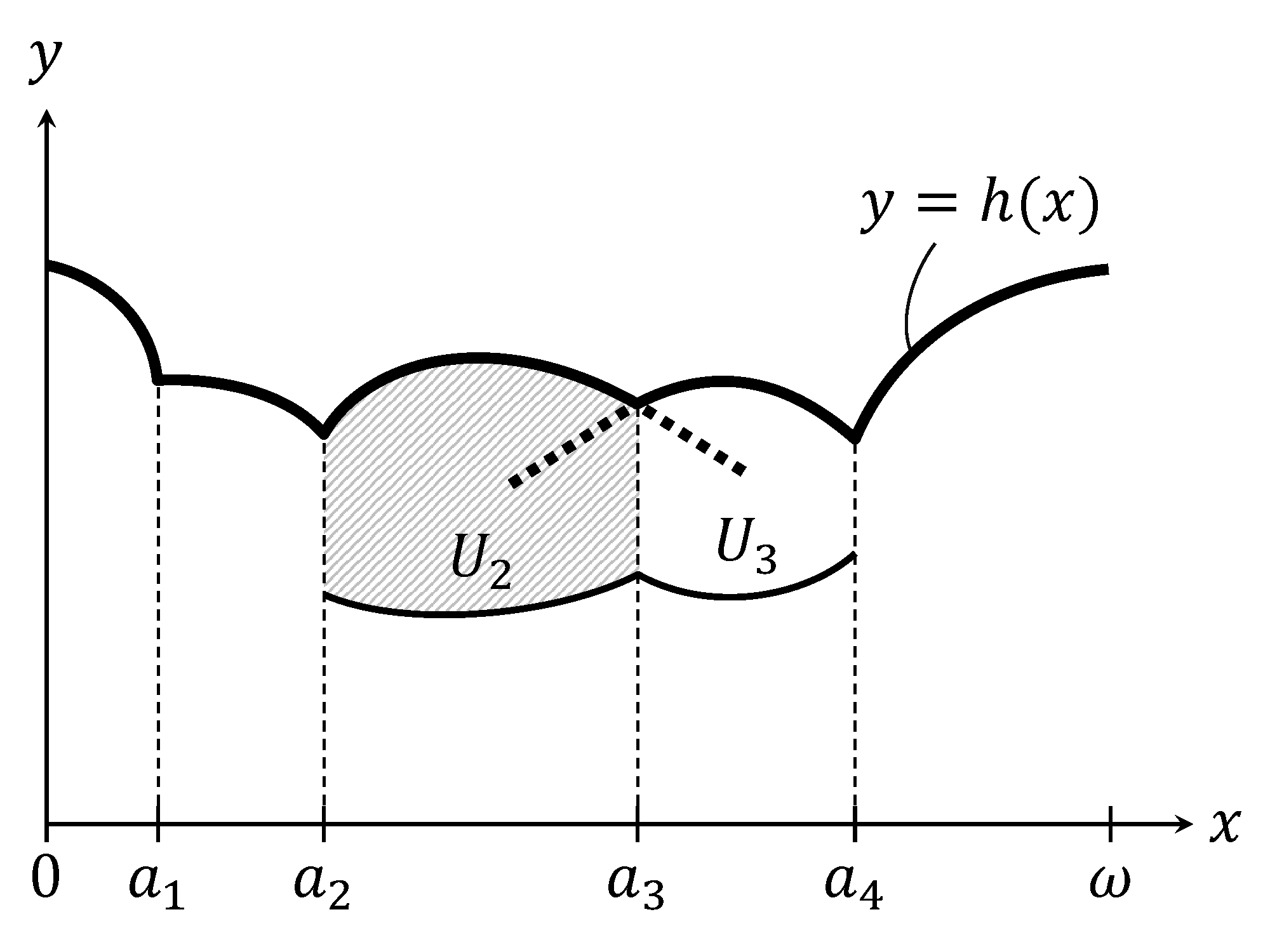}
\caption{bubbles and the graph} \label{Fb}
\end{figure}
The angle condition is denoted as 
\begin{equation} \label{Ean}
	h_x(a_k+0) = h_x(a_k-0) + c_k \quad
	k = 1, 2, \ldots, K,
\end{equation}
where $c_k$ is a given constant and $h_x$ denotes the derivative of $h$;
 here $g(x\pm0)$ is the directional limit defined by
\[
	g(x\pm0) = \lim_{\delta\downarrow0} g(x\pm\delta).
\]
We rather consider that bubbles are moving by a kind of relaxation dynamics.
 Specifically, we postulate that $h=h(x,t)$ is assumed to fulfill a diffusion equation
\[
	\tau h_t = \sigma_1 h_{xx} + A \quad\text{in}\quad
	(a_k, a_{k+1}),\ k=1, \ldots, K,
\]
where $\tau>0$ is a relaxation parameter and $\sigma_1>0$ corresponds a surface tension.
 Here a constant $A$ is to be determined later.
 If we include \eqref{Ean}, then our equation on $\mathbf{T}=\mathbf{R}/(\omega\mathbf{Z})$ becomes
\begin{equation} \label{Eh1}
	\tau h_t = \sigma_1 h_{xx} - f, \quad
	f = \sum_{k=1}^K c_k \delta(x-a_k) - A,
\end{equation}
where $\delta$ denotes Dirac's delta function.
 It is noteworthy that $\int_\mathbf{T}h\,dx$ is conserved if we choose
\begin{equation} \label{Eh2}
	A = \sum_{k=1}^K c_k / \omega.
\end{equation}

On the top of each bubble, we postulate there is a layer of liquid protect a bubble from rupture.
 If thickness of the layer becomes thinner than a given threshold value, we expect the bubble collapse.
 To describe this phenomenon, let $\zeta=\zeta(x,t)$ denote the height of liquid surface so that $\eta=\zeta-h$ describes the (vertical) thickness of the liquid layer;
 see Figure \ref{Fgg}.
\begin{figure}[htb]
\centering
\includegraphics[width=6cm]{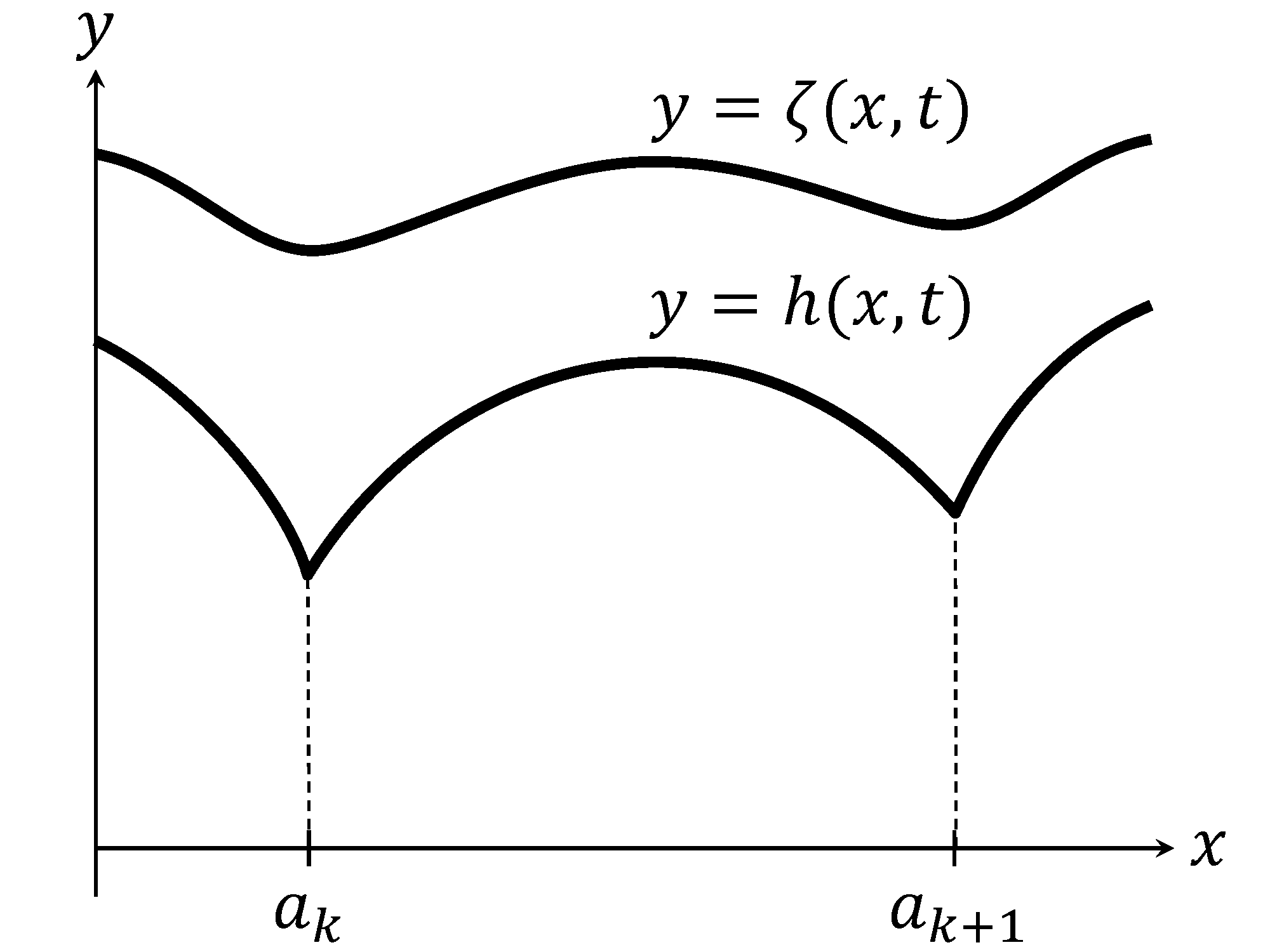}
\caption{functions $\zeta$ and $h$ at time $t$} \label{Fgg}
\end{figure}
We postulate that $\zeta$ satisfies a kind of relaxation dynamics
\begin{equation} \label{Eze}
	\zeta_t = \sigma_2 \zeta_{xx} - \alpha(\zeta-h), \quad
	\alpha \geq 0, \quad \sigma_2>0
\end{equation}
so that $\zeta$ would like to become flat.
 The last term including $\alpha$ describe evaporation of a liquid.

For a given threshold value $\eta_c>0$, we postulate that if $\eta=\zeta-h$ decreases to the value $\eta_c$ at the first time $t_0$ in some point $[a_k,a_{k+1})$, then we restart $\eta$ at $t_0+0$ as
\begin{equation} \label{Eth1}
		\eta(x,t_0+0) = \left \{ 
		\begin{array}{ll}
			\eta_a \quad\text{for} \quad x \in [a_k,a_{k+1}), \\
			\eta(x,t_0-0) \quad\text{for} \quad x \not\in [a_k,a_{k+1}).
		\end{array}
		\right.
\end{equation}
Here $\eta_a$ is a given positive constant larger than $\eta_c$.
 Moreover, we set 
\begin{equation} \label{Eth2}
		h(x,t_0+0) = \left \{
		\begin{array}{l}
			h(x,t_0-0) -d \quad\text{for} \quad x \in [a_k,a_{k+1}), \\
			h(x,t_0-0).
		\end{array}
		\right.
\end{equation}
Here $d>0$ is a fixed number and it roughly describes the vertical side length of the collapsed bubble;
 see Figure \ref{Frup}.
\begin{figure}[htb]
\centering
	\begin{minipage}[b]{0.49\linewidth}
\centering
\includegraphics[width=5.8cm]{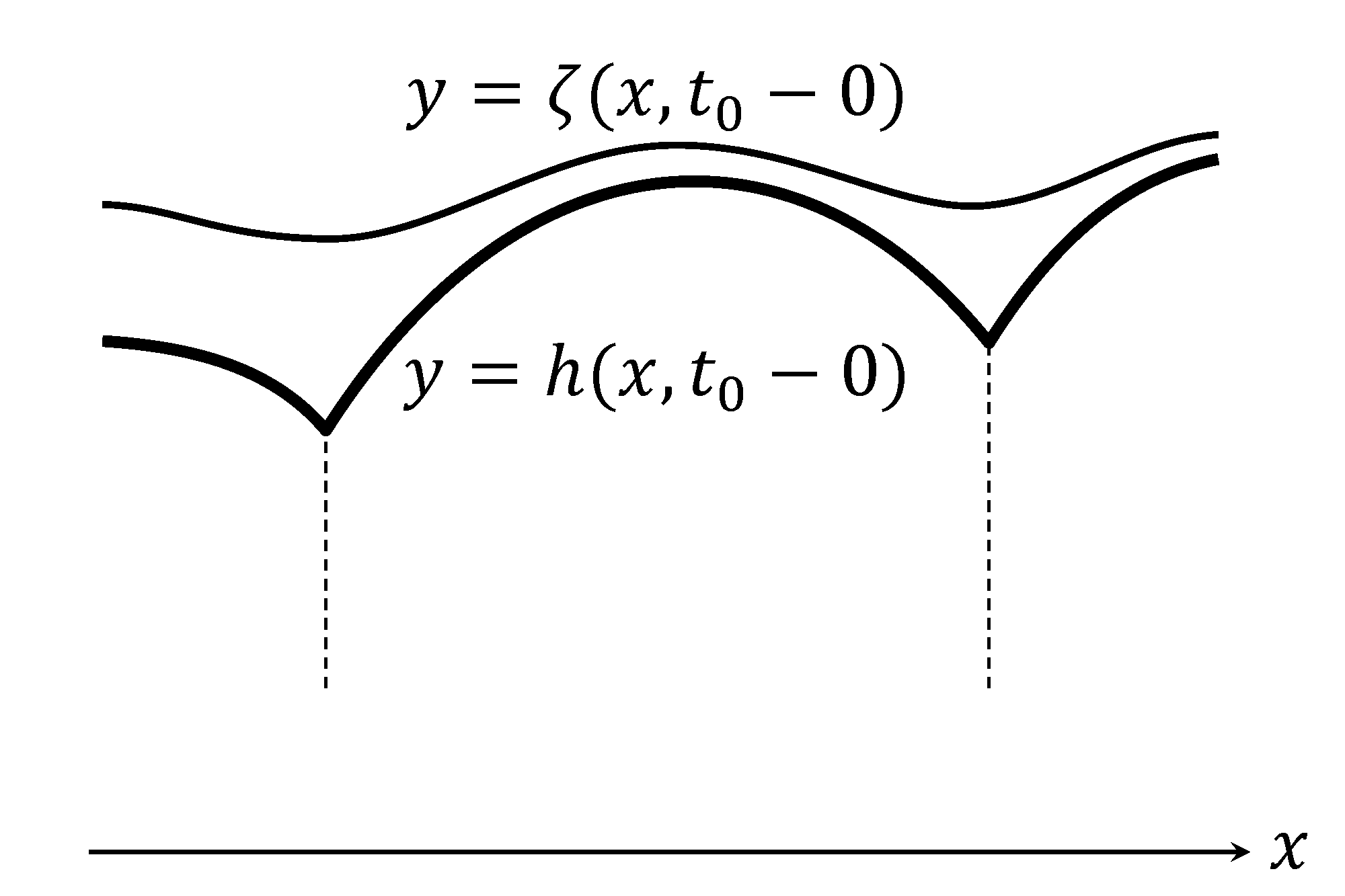}
	\end{minipage}
	\begin{minipage}[b]{0.49\linewidth}
\centering
\includegraphics[width=5.8cm]{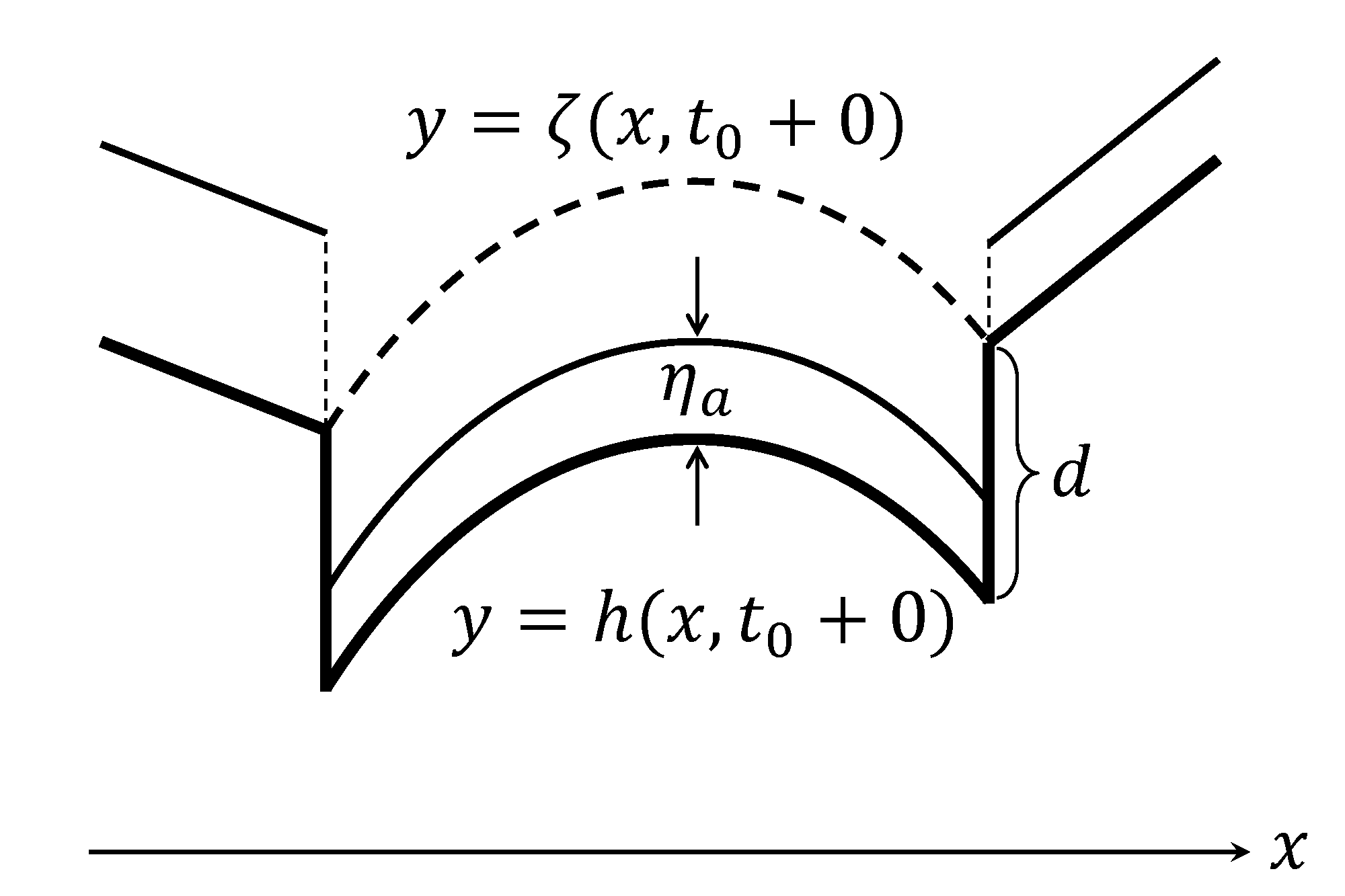}
	\end{minipage}
\caption{at the time of rupture} \label{Frup}
\end{figure}
%

\subsection{Goal of this paper} \label{S1S2} 

We are interested in whether there exists a solution whose profile is time periodic under the assumption that $c_k>0$ for all $k=1,2,\ldots,K$.
 Our first goal is to prove that there exists a time-periodic solution $\eta$ representing the thickness if (i) $\sigma_1/\tau=\sigma_2$ or (i\hspace{-0.1em}i) $\tau=\infty$ under conditions which guarantee that a rupture occurs only in one fixed subinterval $(a_i, a_{i+1})$.
 The time rupture occurs is called a rupture time (cf.\ Definition \ref{Drup}).
 In the cases (i), (i\hspace{-0.1em}i), the equation is reduced to a single equation for $\eta$.
 The existence can be proved by applying a Schauder's fixed point theorem (cf.\ \cite[Theorem 11.1]{GT}) for a mapping from a profile at some rupture time to the profile of the next rupture time.

The second goal is to give some numerical experiments and observe the role of parameter $\sigma_1$, $\sigma_2$ by fixing $\tau=1$ which may not satisfy $\sigma_1/\tau=\sigma_2$.
 From numerical experiments, there may not exist a time-periodic solution when $\sigma_1/\tau=\sigma_2$ is not fulfilled.
 We consider the case $\sigma_1=\sigma_2$ with $\tau=1$.
 In the case when the existence of periodic solution is proved, we check its stability numerically.
 If the dynamics is strongly order-preserving, the uniqueness of a periodic solution can be proved by an abstract theory \cite{OM}.
 Although, the equation for $\eta$ has an order-preserving property, the theory of \cite{OM} does not apply to our setting because of ruptures.
 From numerical experiments, it looks that our periodic solution is globally stable.
 In the case a rupture may happen in several subintervals, we check numerically whether a periodic solution may exist.
 From numerical experiments, we conjecture that there is a chance that there is no periodic solution.

Here is a sketch of the proof of the existence of a periodic solution.
 To apply the fixed point theory, we first check that the first rupture time is well controlled both from above and below.
 This is also useful to prove the compactness of the mapping since there exists a still regularizing effect for the equation of $\eta$ even if there is a singular term like $\delta(x-a_k)$.
 To show the compactness of the mapping is important step and this is done by an explicit formula of a solution of the heat equation with an extend force term.

\subsection{Related problems} \label{S1S3} 

As far as the authors know, there are few papers in mathematical community handling evolution of bubble clusters with rupture.
 In \cite{SS}, a numerical way is given to calculate evolution of bubble clusters with rupture and drainage as well as rearrangement.
 Their bubbles moves by the Navier-Stokes equations but they assume each bubble has a microscopic layer of liquid.
 If this layer becomes thinner than critical thickness, a rupture occurs.
 The evolution of liquid layer is assumed to satisfy a thin-film model of a fluid.
 In their model, inner bubble may rupture.
 In our model, we consider a simple setting so that it can be handled in a rigorous way.
 We consider the situation where rupture only occurs on the boundary of the domain where the cluster occupies;
 see Figure \ref{Fclu}.
\begin{figure}[htb]
\centering

\includegraphics[width=5cm]{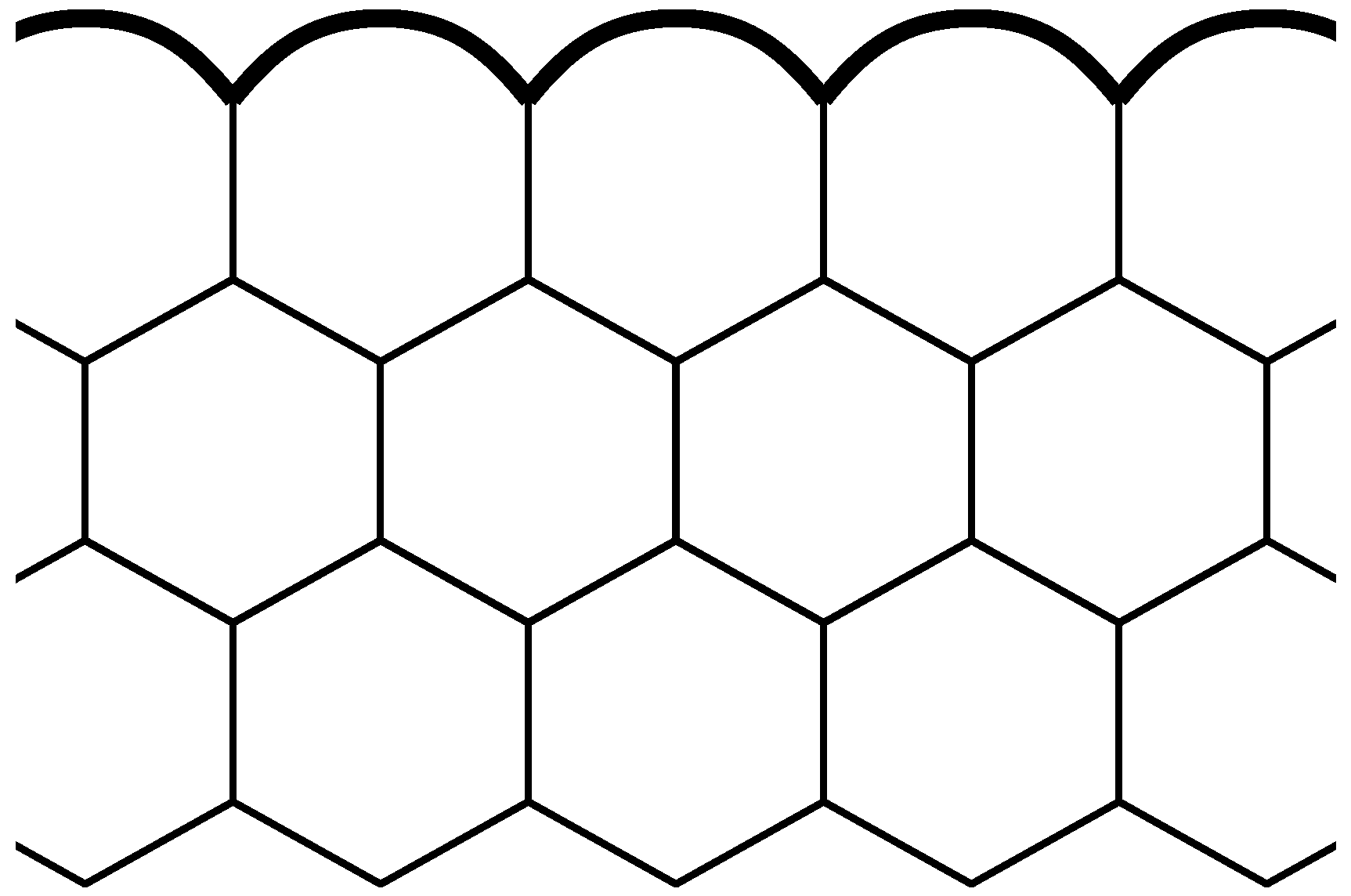}
\caption{clusters} \label{Fclu}
\end{figure}
It is assumed to located at the top in Figure \ref{Fclu}.
 Instead of working thin film equations, we assume that it is just a diffusion equation for $\zeta$.
 The motion of cluster itself is simplified.
 It just slowly moves by a relaxation dynamic like a curvature flow.
 In our setting, we only assume that the top is moving by a diffusion equation.

If one only considers evolution of a cluster of bubbles just by a relaxation dynamics like motion by mean curvature (without liquid layer), it is well studied both as a strong solution (see e.g.\ \cite{MNT}) or a weak solution allowing disappearance of some bubbles (see e.g.\ \cite{EO}).
 This model is often called a multi-grain model and each bubble is called a grain.
 The process including disappearance of some grains is often called a coarsening phenomenon.
 However, this phenomenon is not an occurrence of rupture since the grain boundary does not suddenly disappear.

\subsection{Organization of this paper} \label{S1S4} 
This paper is organized as follows.
 In Section \ref{S2}, we give definition of a rupture time for thickness $\eta$.
 We estimate a rupture time both from above and below when $\alpha>0$ when the system \eqref{Eh1} and \eqref{Eze} is reduced to a single equation for $\eta$.
 In Section \ref{S3}, we prove that a periodic solution exists under some condition which guarantees that a rupture occurs only in one fixed subinterval $(a_i,a_{i+1})$.
 In Section \ref{S4}, we give numerical results.

\section{A single diffusion equation with rupture} \label{S2} 

We are interested in the evolution of thickness $\eta$ when $\zeta$ and $h$ satisfies \eqref{Eze} and \eqref{Eh1} with \eqref{Eh2}, respectively.
 By a direct calculation, we see that
\begin{align*}
	\eta_t = (\zeta-h)_t &= \sigma_2 \zeta_{xx} - \alpha\eta - \frac{\sigma_1}{\tau} h_{xx} + \frac{f}{\tau} \\
	&= \sigma_2 \eta_{xx} - \alpha\eta + \frac{f}{\tau} + \left(\sigma_2 - \frac{\sigma_1}{\tau} \right) h_{xx}.
\end{align*}
In the case $\tau=0$ so that $h_{xx}=f/\sigma_1$, we proceed
\[
	\eta_t = \sigma_2 \eta_{xx} - \alpha \eta + \sigma_2 h_{xx}
	= \sigma_2 \eta_{xx} - \alpha \eta + \frac{\sigma_2}{\sigma_1} f.
\]
With a special choice of parameters, the evolution equation for $\eta$ is decoupled.
\begin{prop} \label{Peq}
Assume that $\zeta$ and $h$ satisfy \eqref{Eze} and \eqref{Eh1}
respectively.
 Let $\eta=\zeta-h$.
\begin{enumerate}
\item[(i)] If $\sigma_2=\sigma_1/\tau$ in \eqref{Eze} and \eqref{Eh1}, then
\[
	\eta_t = \sigma_2 \eta_{xx} - \alpha\eta + \frac{f}{\tau}.
\]
\item[(i\hspace{-0.1em}i)] If $\tau=0$ in \eqref{Eh1}, then
\[
	\eta_t = \sigma_2 \eta_{xx} - \alpha\eta + \frac{\sigma_2}{\sigma_1} f.
\]
\end{enumerate}
\end{prop}

In the case $\sigma_2=\sigma_1/\tau$ or $\tau=0$, our problem is reduced to the evolution equation of the form
\begin{equation} \label{Eeta}
	\eta_t = \sigma\eta_{xx} - \alpha\eta + f, \quad
	f = \sum_{k=1}^K c_k \delta(x-a_k) - A
\end{equation}
with $\sigma=\sigma_2$ and possibly different values of $c_k$'s.
 Its stationary solution (time-independent solution) $\eta=s(x)$ is easy to find.
 It must satisfy
\begin{align*}
	&\sigma s_{xx} - \alpha s = A \quad\text{in}\quad (a_k, a_{k+1}), \\
	&s_x (a_k+0) = s_x(a_k-0) - c_k/\sigma \quad\text{for}\quad k = 1,\ldots,K.
\end{align*}
In particular, $s$ is 
of the form
\[
	s(x) = -\frac{A}{\alpha} + b_+ e^{\lambda x} + b_- e^{-\lambda x}, \quad
	\lambda = \sqrt{\alpha/\sigma}
\]
with constants $b_\pm$
in each $(a_k,a_{k+1})$ whose derivative jumps at $a_k$.
 Such a solution uniquely exists for $x>0$; see Figure \ref{Fst}.
\begin{figure}[htb]
\centering
\includegraphics[width=7cm]{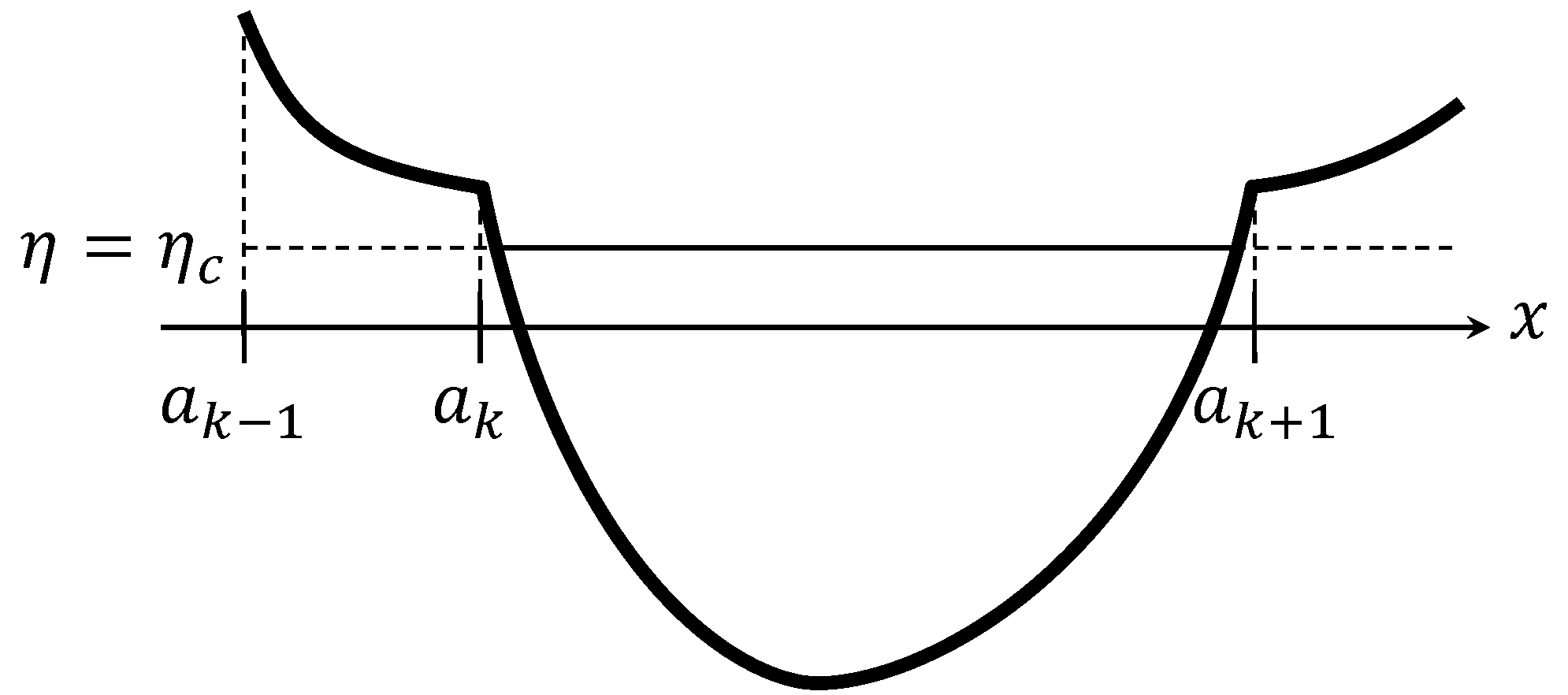}
\caption{graph of $s$} \label{Fst}
\end{figure}
 If $\alpha=0$, a stationary solution exists if and only if \eqref{Eh2} holds and it is quadratic in each $(a_k,a_{k+1})$.
 If we consider $\bar{\eta}=\eta-s$, then $\bar{\eta}$ satisfies
\[
	\bar{\eta}_t = \sigma\bar{\eta}_{xx} - \alpha\bar{\eta}.
\]
Since we impose the periodic boundary condition, in the case $\alpha=0$, $\bar{\eta}$ does not converge to zero as $t\to\infty$ provided that the average of $\bar{\eta}$ on $[0,\omega)$ is not zero.
 This indicates that there is a chance that collapse of bubbles may not occur when $\alpha=0$.
 If the evaporation parameter $\alpha>0$, then bubbles may collapse.
 For later convenience, we give a definition of rupture time.
\begin{definition} \label{Drup}
Let $\eta_c>0$ be a threshold value.
 Let $\eta$ be a solution of \eqref{Eeta} with initial data $\eta_0=\eta(x,0)>\eta_c$.
 The time
\[
	t_r(\eta_0) := \sup \left\{ t \biggm|
	\inf_{x\in\mathbf{T}} \eta(x,t) > \eta_c \right\}
\]
is said to be a \emph{rupture time} for $\eta_0$.
\end{definition}
\begin{prop} \label{Prup1}
Assume that $\int_0^\omega f\,dx\leq0$.
 Then, the rupture time $t_r$ is finite provided that $\alpha>0$.
 Moreover,
\[
	t_r \leq - \frac1\alpha \log \left( \frac{\eta_c}{\frac1\omega \int_0^\omega \eta_0\, dx} \right).
\]
\end{prop}
\begin{proof}
We integrate \eqref{Eeta} on $(0,\omega)$ to get
\[
	\frac{d}{dt} \int_0^\omega \eta\, dx \leq -\alpha \int_0^\omega \eta\, dx
\]
i.e.,
\[
	\frac1\omega \int_0^\omega \eta\,dx
	\leq e^{-\alpha t} \left( \frac1\omega \int_0^\omega \eta_0\, dx \right).
\]
If the average $\frac1\omega \int_0^\omega \eta\,dx(t)\leq\eta_c$, we see $t_r\leq t$.
 The desired estimate now follows.
\end{proof}

We next give an estimate of the rupture time from below.
\begin{prop} \label{Prup2}
Assume that $c_k$ is nonnegative for all $k=1,2,\ldots,K$ and $A\geq0$.
 If $\inf\eta_0>\eta_c$, then
\[
	t_r \geq t_* := \frac1\alpha \log \left[ \left(\frac{A}{\sigma} + \inf\eta_0\right) \biggm/
	\left(\frac{A}{\alpha} + \eta_c \right) \right] > 0.
\]
\end{prop}
\begin{proof}
Let $\xi$ be a solution of an ODE
\[
	\xi_t - \alpha\xi + A = 0.
\]
Since $c_k\geq0$ for all $k$,
this $\xi$ is a subsolution of \eqref{Eeta}.
 We consider $\xi$ with initial data $\xi(0)=\inf\eta_0$.
 By the comparison principle, we observe that
\[
	\eta(x,t) \geq \xi (t).
\]
Thus, we have
\[
	t_r(\eta_0) \geq t_* := \sup \left\{ t > 0 \biggm|
	\inf_{(0,t)} \xi(\tau) > \eta_c \right\}.
\]
Since
\[
	\xi(t) = A \left( \frac{e^{-\alpha t}-1}{\alpha} \right)
	+ ( \inf \eta_0) e^{-\alpha t},
\]
$t_*$ has an explicit form
\[
	t_* = \frac1\alpha \log \left[ \left(\frac{A}{\alpha} + \inf\eta_0\right) \biggm/
	\left(\frac{A}{\alpha} + \eta_c \right) \right].
\]
\end{proof}

\section{Existence of a periodic solution} \label{S3} 

We consider the evolution of $\eta$ by \eqref{Eeta} with rupture.
 Here is a precise form of the dynamics.
 We consider $\eta$ of \eqref{Eeta} with initial data $\eta_0>\eta_c$, where $\eta_c>0$ is a given threshold value.
 At the first rupture time $t_1$, let $R=R(t_1)$ be the set such that
\[
	R(t_1) = \left\{ x \in \mathbf{T} \bigm|
	\eta(x,t_1) = \eta_c \right\}.
\]
We call this set the \emph{rupture set}.
 Let $\eta_a$ be a given number satisfying $\eta_a>\eta_c$.
 We set
\[
	\eta(x, t_1+0) = \eta_a
\]
for $x\in[a_k,a_{k+1})$ if $[a_k,a_{k+1})\cap R$ is not empty.
 For other $x$, we set
\[
	\eta(x, t_1+0) = \eta(x, t_1-0).
\]
For $t>t_1$, let $\eta$ be the solution of \eqref{Eeta} such that the value at $t=t_1$ equals $\eta(x,t_1+0)$ defined above.
 Let $t_r$ be the rupture time with initial data $\eta(x,t_1+0)$.
 We set $t_2=t_1+t_r$, which is the second rupture time.
 We repeat the same modification of $\eta$ at $t=t_2$ and proceed further until the third rupture time.
 We modify $\eta$ again at that time.
 We say that the resulting $\eta$ is a \emph{solution of \eqref{Eeta} with rupture} (fixing $\eta_a$ and $\eta_c$ with $\eta_a>\eta_c$).

We are interested in the behavior of $\eta$.
 Proposition \ref{Prup1} says that there is infinite many rupture times.
 However, it is not clear whether or not the set of 
 rupture times
 is discrete.
 To simplify the situation, we give a sufficient condition that the rupture set $R$ is contained in a single fixed interval $(a_i, a_{i+1})$.
\begin{enumerate}
\item[(S)] 
Let $s$ be a stationary solution
of \eqref{Eeta}.
 Assume that $\eta_c$ satisfies $s>\eta_c$ outside single interval $(a_i, a_{i+1})$.
 See Figure \ref{Fst} with $k=i$.
 Moreover, we assume that $\eta_a(>\eta_c)$ satisfies $s<\eta_a$ on $[a_i, a_{i+1}]$.
\end{enumerate}
%
For later convenience, we collect assumptions so that Proposition \ref{Prup1} and \ref{Prup2} apply.
 We shall assume
\begin{enumerate}
\item[(C)] the constant $c_k$ is nonnegative for all $k=1,2,\ldots,K$ and $A\geq\sum_{k=1}^K c_k/\omega$ so that $\int_0^\omega f\,dx\leq0$, where $f=\sum_{k=1}^K c_k \delta(x-a_k)-A$.
\end{enumerate}
 As the next lemma indicates, 
 these assumptions guarantee
 that the bubble collapse occurs only on $(a_i, a_{i+1})$.
\begin{lemma} \label{Lsing}
Assume (C) and (S).
 Let $\eta$ be the solution of \eqref{Eeta} with rupture where initial data $\eta_0>\max(s,\eta_c)$.
 Then the times $\{t_j\}$ when $\eta$ ruptures are ordered as
\[
	0 < t_1 < t_2 < \cdots < t_j < \cdots
\]
and $\lim_{j\to\infty}t_j=\infty$.
 Moreover, the rupture set $R$ at each rupture time is included in $(a_i,a_{i+1})$.
\end{lemma}
\begin{proof}
By a comparison principle, $\eta>s$ for $t\in(0,t_1)$.
 By (S), we have $\eta>\eta_c$ outside $[a_i,a_{i+1}]$.
 Thus $R(t_1)$ is included in $(a_i,a_{i+1})$.
 At $t=t_1$, $\eta(x,t_1+0)>s(x)$ for $x\in\mathbf{T}$ because of the assumption on $\eta_a$ in (S).
 Moreover, 
\[
	\inf_x \eta(x,t_1+0) - \eta_c
	> \inf_{x\not\in(a_i,a_{i+1})} s - \eta_c =: \rho > 0
\]
by (S).
 Let $t_2$ denote the second rupture time.
 By Proposition \ref{Prup2},
\[
	t_2 - t_1 \geq t_* > 0,
\]
where $t_*$ is a constant depending only on $\sigma$, $\eta_c$, $\alpha$ and $A$.
 The same procedure implies that $t_j-t_{j+1}\geq t_*>0$.
 The existence of $t_1<t_2<\cdots$ is guaranteed by Proposition \ref{Prup1}.
 As in the first step, at each step the rupture set is always contained in $(a_i, a_{i+1})$.
 The proof is now complete.
\end{proof}
We are now in position to state our main result on existence of a periodic solution.
\begin{thm} \label{Tmain}
Assume (C) and (S).
 There exists a solution $\eta$ of \eqref{Eeta} with rupture and $T>0$ such that
\[
	\eta(x,t+T) = \eta(x,t) \quad\text{for all}\quad
	t \in \mathbf{R},\ x \in \mathbf{T}.
\]
\end{thm}
We shall prove Theorem \ref{Tmain} by applying Schauder's fixed point theorem (cf.\ \cite[Theorem 11.1]{GT}).
 We set
\begin{equation} \label{Eset}
	E := \left\{ \xi \in C\left(\mathbf{T} \backslash (a_i, a_{i+1}) \right) \bigm|
	s \leq \xi \leq s+B\ \text{on}\ \mathbf{T} \backslash (a_i, a_{i+1}) \right\},
\end{equation}
where $B$ is taken so that $\inf_{x\in\mathbf{T}}(s+B)>\eta_a$. 
 For $\xi\in E$, we set
\begin{equation}
		\eta_0 = \left \{
		\begin{array}{l} \label{Eint}
			\eta_a, \quad x \in(a_i, a_{i+1}), \\
			\xi(x), \quad x \not\in(a_i, a_{i+1}).
		\end{array}
		\right.
\end{equation}
We solve \eqref{Eeta} with initial data $\eta_0$.
 Let $\eta$ be the solution of \eqref{Eeta} with initial data $\eta_0$.
 Let $t_r(\xi)$ denote its rupture time.
 By Lemma \ref{Lsing}, the rupture set is contained in $(a_i,a_{i+1})$.
 We define a nonlinear map $\mathcal{T}$ by
\begin{equation} \label{Emap}
	(\mathcal{T}\xi) (x) := \eta \left(x,t_r(\xi)-0 \right)
	\quad\text{for}\quad x \in \mathbf{T} \backslash (a_i, a_{i+1}).
\end{equation}
A fixed point of $\mathcal{T}$ gives a time-periodic solution of \eqref{Eeta} with rupture.
 In the rest of this section, we shall prove the existence of  a fixed point by applying Schauder's fixed point theorem.

To say compactness of the mapping $\mathcal{T}$, we prepare a regularity lemma for the heat equation with irregular inhomogeneous term.
 We begin with a simple estimate of the Gauss kernel $G_t(x)=(4\pi t)^{-n/2}\exp\left(-|x|^2/4t\right)$ for $x=(x_1,\ldots,x_n)\in\mathbf{R}^n$.
 We begin with a well-known estimate.
\begin{lemma} \label{Lgau}
Let $\Gamma(z)$ denote the Gamma function, i.e.,
\[
	\Gamma(z) = \int_0^\infty e^{-r} r^{z-1}\, dr.
\]
 Then
\[
	\sup_{t>0} \sup_{x\in\mathbf{R}^n}
	\int_0^t \left|\nabla G_\tau(x)\right|d\tau\cdot |x|^{n-1}
	= \Gamma \left(\frac{n}{2}\right) \pi^{-\frac{n}{2}}.
\]
\end{lemma}
\begin{proof}
We give here a proof for completeness.
 Since
\[
	\frac{\partial}{\partial x_i} G_\tau(x)
	= - \frac{x_i}{2\tau} G_\tau(x),
\]
we proceed
\[
	\int_0^t \left|\nabla G_\tau(x)\right|d\tau
	\leq \int_0^t \frac{|x|}{2\tau} \frac{1}{(4\pi\tau)^\frac{n}{2}} \exp \left(- \frac{|x|^2}{4\tau}\right) d\tau.
\]
By charging the variable $\tau$ of integration by $\tau=|x|^2/4r$, the right-hand side equals
\[
	\pi^{-\frac{n}{2}} \int_{|x|^2/4r}^\infty e^{-r} r^{1+\frac{n}{2} -2}\, dr |x|^{-1-n+2}
\]
since $d\tau=-\left(|x|^2/2r^2\right)dr$.
 This quantity is estimated from above by $\pi^{-n/2}\Gamma(n/2)|x|^{1-n}$.
\end{proof}

For $\eta_0\in C(\mathbf{T})$, let $\eta$ be the solution of \eqref{Eeta} with initial data $\eta_0$.
 Let $S(t):\eta_0\longmapsto\eta(\cdot,t)$ denote the solution operator so that $\eta(x,t)=\left(S(t)\eta_0\right)(x)$.
\begin{lemma} \label{Lreg}
There is a numerical constant $C$ such that for each $T>0$ there is a constant $C_T$ depending only on $T$ and $\omega$ satisfying
\[
	\left\| \partial_x S(t) \eta_0 \right\|_\infty
	\leq \frac{C}{t^{1/2}} \|\eta_0\|_\infty
	+ C_T \sum_{k=1}^K |c_k|
\]
for all $t>0$, $\eta_0\in C(\mathbf{T})$, where $\|\cdot\|_\infty$ denotes the sup norm on $\mathbf{T}$.
\end{lemma}
\begin{proof}
We note that
\begin{equation} \label{Esol}
	S(t) \eta_0 = e^{tL} \eta_0
	+ \int_0^t e^{(t-\tau)L} f\, d\tau, \quad
	f = \sum_{k=1}^K c_k \delta(x-a_k) - A,
\end{equation} 
where $L=\Delta-\alpha$ so that
\[
	(e^{tL} \eta_0) (x)
	= (G_t^\alpha * \eta_0) (x)
	= \int_\mathbf{R} G_t^\alpha(x-y) \eta_0(x)\, dx, \quad
	G_t^\alpha := e^{-\alpha t} G_t.
\]
Here $\eta_0$ is regarded as a periodic function on $\mathbf{R}$ with period $\omega$.
 The estimate
\[
	C := \sup_{t>0} t^{1/2} \| \partial_x e^{tL} \eta_0 \|_\infty \bigm/ \| \eta_0 \|_\infty < \infty
\]
is standard and known as an $L^\infty$-$L^\infty$ estimate for derivative.
 It suffices to estimate the term $\partial_x\int_0^t e^{(t-s)L}f\,d\tau$.
 Since $\partial_x(e^{tL}A)=0$ and $|e^{-\alpha t}|\leq1$, it suffices to estimate
\[
	I(x) = \sum_{j=-\infty}^\infty \partial_x 
	\int_0^t \int_\mathbf{R} G_{t-\tau} (x-y) \delta(y-j\omega-a_k)\, dyd\tau
\]
for each $k=1,\ldots,K$.
 We may assume that $a_k=0$.
 We may also assume that $\omega=2$ by scaling.
 It suffices to estimate
\[
	I(x) = \sum_{j=-\infty}^\infty 
	\int_0^t \partial_x G_{t-\tau} (x-j) \, d\tau
\]
for $x\in[-1,1]$.
 We proceed
\[
	I(x) = \partial_x 	\int_0^t G_{t-\tau} (x)\, d\tau
	+ \sum_{j\neq0} \int_0^t (\partial_x G_{t-\tau}) (x-2j)\, d\tau
	= I_1 + I_2.
\]
The first term $I_1$, which is the leading term is estimated as
\[
	\left| I_1(x) \right| \leq 1 (= \pi^{-1/2} \Gamma(1/2))
\]
by Lemma \ref{Lgau}.
 We shall estimate $I_2(x)$ for $|x|\leq1$.
 Since $xe^{-x^2}\leq(2e)^{-1/2}$ for $x>0$ and $\partial_x G_t=-(x/2t)G_t$, we see
\[
	| \partial_x G_t | \leq \frac{C_0}{t^{1/2}} G_{t/2}
\]
with some constant $C_0$ independent of $t$ and $x$.
 Thus,
\begin{align*}
	\sup_{|x|\leq1} \left| I_2(x) \right|
	& \leq \sum_{j\neq0} \int_0^t \frac{1}{(t-\tau)^{1/2}} G_{(t-\tau)/2} (x-2j)\, d\tau \\
	& \leq 2 \sum_{j=1}^\infty \int_0^t \frac{1}{r^{1/2}} G_{r/2} (2j-1)\, d\tau.
\end{align*}
Since $(2j-1)^2\geq j$ for $j=1,2,\ldots$, we see
\[
	\sum_{j=1}^\infty G_{r/2} (2j-1)
	\leq \sum_{j=1}^\infty G_{r/2} (1) e^{-j/4r}
	= G_{r/2} (1) \frac{e^{-1/(4r)}}{1-e^{-(1/4r)}}.
\]
Thus
\[
	\sup_{|x|\leq1} \left| I_2(x) \right|
	\leq \int_0^t \frac{1}{r^{1/2}} G_{r/2} (1)\, dr \frac{1}{1-e^{-(1/4T)}}.
\]
The integrand $r^{-1/2}G_{r/2}(1)$ is integrable on $(0,T)$ so we have a bound of $\sup_{|x|\leq1} \left| I_2(x) \right|$ for $t\leq T$.
 We now conclude that $\sup_{0\leq t\leq T} \sup_{|x|\leq1} I(x):=C_T$ is finite.
 The proof is now complete.
\end{proof}
\begin{proof}[Proof of Theorem \ref{Tmain}]
We consider the mapping $\mathcal{T}$ defined by \eqref{Emap} on the set $E$ defined by \eqref{Eset}.
 To have a fixed point by Schauder's theorem, it suffices to prove that
\begin{enumerate}
\item[(i)] $E$ is a convex, closed set in a Banach space $C\left(\mathbf{T}\backslash(a_i,a_{i+1})\right)$;
\item[(i\hspace{-1pt}i)] the mapping $\mathcal{T}$ is continuous from $E$ into $E$;
\item[(i\hspace{-1pt}i\hspace{-1pt}i)] its image $\mathcal{T}(E)$ is relatively compact.
\end{enumerate}
The convexity of $E$ is clear by definition.
 The closedness of $E$ under the sup-norm in $C\left(\mathbf{T}\backslash(a_i,a_{i+1})\right)$ is also clear.
 By the solution formula \eqref{Esol}, it is not difficult to see that $\eta\in C\left(\mathbf{T}\times(\delta,\infty)\right)$ for any $\delta>0$.
 Since $s+B$ is a supersolution of \eqref{Eeta}, we see $S(t)\eta_0\leq s+B$ for $\eta_0$ defined by \eqref{Eint} if we choose $B$ such that $\inf_{x\in\mathbf{T}}(s+B)>\eta_a$.
 Since $S(t)\eta_0\geq s$ for $\eta_0\geq s$, we now observe that $s\leq S(t)\eta_0\leq s+B$.
 By the maximum principle, we know
\[
	\left\| S(t)\eta_{01} - S(t)\eta_{02} \right\|_{L^\infty(\mathbf{T})}
	\leq \| \eta_{01}-\eta_{02} \|_{L^\infty(\mathbf{T})}
\]
which implies that $t_r(\xi)$ moves continuously in $\xi\in E$.
 Thus, the mapping $\mathcal{T}$ is continuous from $E$ into $E$ so we obtain (i) and (i\hspace{-0.1em}i).
 Up to this moment, we only use the assumption that $\int_0^\omega f\,dx=0$ for $f=\sum_{k=1}^K c_k-A$ so that the rupture time exists.

It remains to prove that $\mathcal{T}(E)$ is relatively compact.
 By our assumption (S), if we take
\begin{equation*}
		\eta_0 = \left \{
		\begin{array}{l}
			\eta_a, \quad x \in(a_i, a_{i+1}) \\
			\xi(x), \quad x \not\in(a_i, a_{i+1})
		\end{array}
		\right.
\end{equation*}
for $\xi\in E$, then $\inf\eta_0\geq\eta_c+\delta$ for some $\delta>0$ independent of $\xi\in E$.
 Since we assume (C), applying Proposition \ref{Prup2}, we observe that $t_r(\xi)\geq\delta'$ with some $\delta'>0$ independent of $\xi\in E$.
 This estimate for $t_r(\xi)$ from below is crucial in our proof.
 The estimate for $t_r(\xi)$ from above is obtained by Proposition \ref{Prup1}.
 The estimate in Proposition \ref{Prup1} implies that there is $T$ depending only on $E$ and the equation \eqref{Eeta} such that
\[
	t_r(\xi) \leq T
\]
for all $\xi\in E$.
 We now apply Lemma \ref{Lreg} to get
\[
	\left\| \partial_x \mathcal{T} (\xi) \right\|'_\infty
	\leq \frac{C\|\xi\|'_\infty}{(\delta')^{1/2}}
	+ C_T \sum_{k=1}^\infty c_k,
\]
where $\|\cdot\|'_\infty$ is the sup norm on $\mathbf{T}\backslash(a_i,a_{i+1})$.
 Since $\mathcal{T}(E)$ is bounded, we are able to apply the Arzel\`a-Ascoli theorem to conclude that $\mathcal{T}(E)$ is relatively compact in $C\left(\mathbf{T}\backslash(a_i,a_{i+1})\right)$.
 The proof is now complete.
\end{proof}
\noindent
\textbf{Generalization.}
 The proof of existence of a periodic-in-time solution can be easily generalized in more general setting.
 Let $\Omega$ be a metric measure space.
 We consider a family $\left\{S(t)\right\}_{t\geq0}$ is an order-preserving semigroup which is $*$-weakly continuous at $t=0$.
 In other words, we assume
\begin{enumerate}
\item[(S1)] (semigroup property)
 $S(t)S(\tau)\eta_0=S(t+\tau)\eta_0$. $t,\tau\geq0$, $\eta_0\in L^\infty(\Omega)$ and $S(0)$ is the identity operator;
\item[(S2)] (convergence at time zero)
 $S(t)\eta_0$ converges to $\eta_0$ as $t\downarrow0$ in the $*$-weak topology of $L^\infty(\Omega)$;
\item[(S3)] (order preserving)
 if $\eta_0\geq\eta_1$ in $L^\infty(\Omega)$, then $S(t)\eta_0\geq S(t)\eta_1$ for all $t\geq0$.
\end{enumerate}
For a stationary solution $s$, we assume
\begin{enumerate}
\item[(S4)] (unique existence of a stationary solution and its global stability)
 there is a unique $s\in C(\overline{\Omega})\cap L^\infty(\Omega)$ such that $S(t)s=s$ for all $t>0$.
 Moreover, $S(t)\eta_0\to s$ in $L^\infty(\Omega)$ as $t\to\infty$.
\end{enumerate}
Let $\eta_c>0$ be a given threshold value.
 For the stationary solution, we further assume
\begin{enumerate}
\item[(S5)] there is an open set $D\subset\Omega$ such that $\inf_{\Omega\backslash D}s>\eta_c$ and $\inf_D s<\eta_c$.
\end{enumerate}
This assumption implies that the rupture occurs only in $D$.
 We consider the evolution with rupture.
 Let $t_r(\eta_0)$ be the rupture time starting with $\eta_0$, i.e.,
\[
	t_r(\eta_0) := \sup \left\{ t \Bigm| \inf_\Omega S(t)\eta_0 > \eta_c \right\}.
\]
By (S4) and (S5), such $t_r(\eta_0)$ exists as a positive number.
 We set
\[
	S^r(t)\eta_0 := S(t)\eta_0, \quad
	0 < t < t_r(\eta_0)
\]
and
\begin{equation*}
		\left( S^r \left( t_r(\eta_0) \right)\eta_0 \right)(x) := \left \{
		\begin{array}{ll}
			\eta_a &\text{for}\quad x \in D \\
			\left( S \left( t_r(\eta_0) \right)\eta_0 \right)(x) &\text{for}\quad x \in \Omega\backslash D.
		\end{array}
		\right.
\end{equation*}
Here $\eta_a$ is a positive number such that $\eta_a>\sup_{\overline{D}}s$.
 To guarantee that $t_r$ is bounded from below, we assume that
\begin{enumerate}
\item[(S6)] there is $t_+$ depending only on $\inf \eta_0-\eta_c$ such that $t_r(\eta_0) \geq t_+ > 0$.
\end{enumerate}
We define $S^r(t)$ successively.
 We set $t_1=t_r(\eta_0)$ and set 
\[
	S^r(t)\eta_0 := S(t-t_1)\eta^{t_1}, \quad
	\eta^{t_1} = S^r(t_1)\eta \quad\text{for}\quad t \quad
	\text{satisfying}\quad t_1 \leq t < t_r(\eta^{t_1}) + t_1
\]
and
\begin{equation*}
		\left( S^r (t_2) \eta_0 \right)(x) := \left \{
		\begin{array}{l}
			\eta_a \quad\text{for}\quad x \in D \\
			S \left( t_r(\eta^{t_1}) \right)\eta^{t_1}
			\quad\text{for}\quad x \in \Omega\backslash D,\ 
			t_2 = t_1 + t_r(\eta^{t_1}).
		\end{array}
		\right.
\end{equation*}
For $j=2,3$, we denote $\eta^{t_j}=S^r(t_j)\eta_0$ and define $t_{j+1}=t_j+t_r(\eta^{t_j})$ so that
\[
	S^r(t)\eta_0 := S(t-t_j) \eta^{t_j} \quad\text{for}\quad t, \quad
	t_j \leq t < t_{j+1}
\]
\vspace{-1.4em}
\begin{equation*}
		\left( S^r(t_{j+1})\eta_0 \right)(x) := \left \{
		\begin{array}{l}
			\eta_a \quad\text{for}\quad x \in D \\
			S \left( t_r(\eta^{t_j}) \right)\eta^{t_j}
			\quad\text{for}\quad x \in \Omega\backslash D.
		\end{array}
		\right.
\end{equation*}
By (S3), (S6) together with (S4), we see that $\lim_{j\to\infty}t_j=\infty$.
 It is easy to see that Theorem \ref{Tmain} can be generalized as
\begin{thm} \label{Tgen}
Assume that (S1)--(S6).
 Assume, moreover, that $S(t)(s+B)\leq s+B$ for all $B\geq B_0$ and $t\geq0$ with some $B_0>0$.
 Assume, furthermore, that
\[
	\left\{ S(t)z \bigm| z \in U,\ t>\delta \right\}
\]
is relatively compact in $L^\infty(\Omega)$ for any bounded set $U$ and $\delta>0$.
 Then, there is a periodic-in-time evolution $S^r(t)\eta_*$ with rupture.
 In other words, there exists $T$ and $\eta_*\in L^\infty(\Omega)$ such that
\[
	S^r(t+T) \eta_* = S^r(t)\eta_*
\] 
for all $t>0$.
\end{thm}
\begin{remark} \label{RSt} 
Our idea for proving the existence of a periodic solution based on compactness somewhat resembles that of \cite{ON} and \cite{GIK}.
 The existence of a periodic solution (rotating spiral type solution) for an Allen-Cahn type equation on an annulus was proved in \cite{ON}.
 The existence of a spiral type solution for a forced (weakly anisotropic) curvature flow equation on an annulus was proved in \cite{GIK}.
 In both cases, uniqueness and stability of a periodic solution have been established based on an abstract theory \cite{OM} for the strongly ordered dynamics.
 Unfortunately, our system does not satisfy their assumptions.
 Although there is 
 another abstract
 theory \cite{OHM} 
 for convergence to one of periodic solutions,
  our situation does not seem to fall in their setting because our mapping $\mathcal{T}$ is not order preserving;
 see Figure \ref{Fnot}.
\begin{figure}[htb]
\centering
	\begin{minipage}[b]{0.49\linewidth}
\centering
\includegraphics[width=7cm]{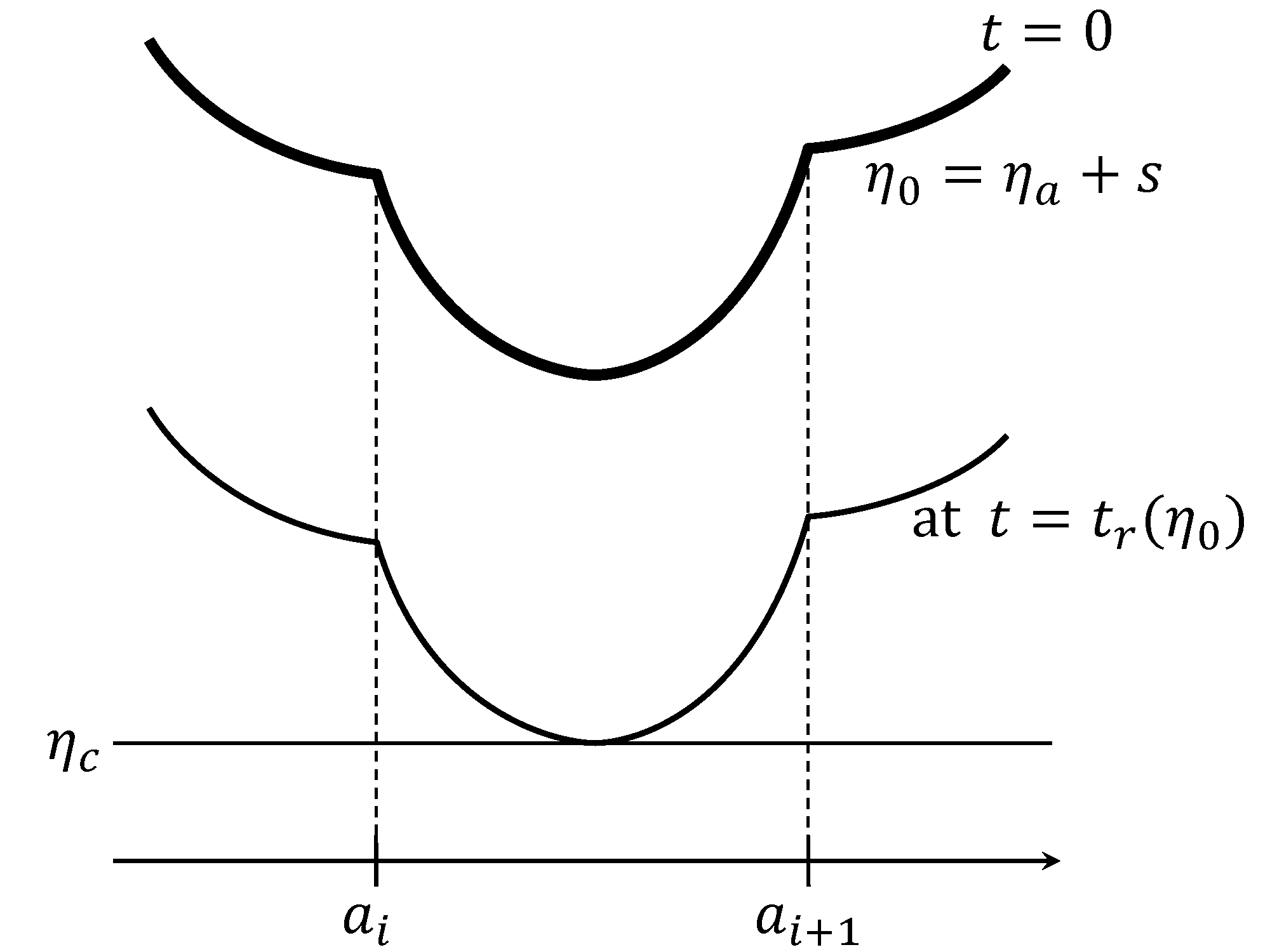}
	\end{minipage}
	\begin{minipage}[b]{0.49\linewidth}
\centering
\includegraphics[width=7cm]{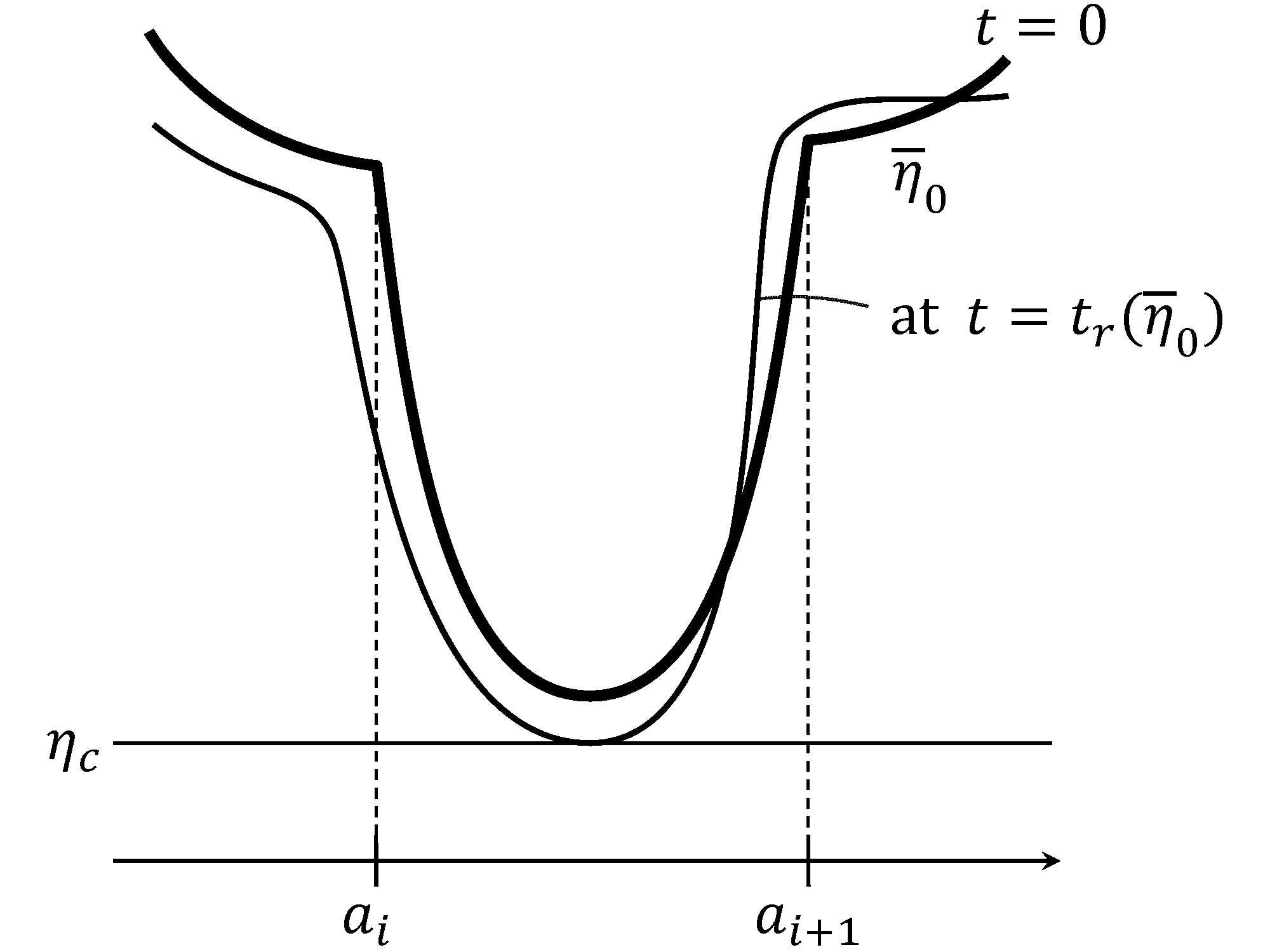}
	\end{minipage}
\caption{Initially $\eta_0\geq\overline{\eta}_0$ but at the rupture time the order is not preserved because $t_r(\eta_0)$ is far larger than $t_r(\overline{\eta}_0)$.} \label{Fnot}
\end{figure}
We do not know the stability of our periodic solution although it is likely by numerical experiment.
\end{remark}

\section{Numerical results} \label{S4}
In this section, we show some numerical results. For simplicity, we only consider the case $\tau=1$, $\omega=1$, $\eta_a = 0.03$, $d = 0.1$ and $c_k = 1$ for all $k=1, 2, \dots, K$. Furthermore, we let $A = \sum_{k=1}^Kc_k/\omega = K$ for all examples, then the assumption (C) is always satisfied. We employ the finite element method with Lagrange P1 element for spatial discretiazation and backward Euler method for temporal discretization.
\begin{exam}\label{Ex1}
The first example demonstrates the periodic solution $\eta$. Let $\sigma_2=\sigma_1/\tau=1$, then we have equation \eqref{Eeta}.  The domain $[0,1)$ is divided by $\{a_k\}_{k=1}^K = \{0.1, 0.6, 0.9\}$. The interval $[a_1,a_2] = [0.1,0.6]$ is significantly longer than the other intervals. Then it is expected that the assumption (S) is satisfied for sufficiently small $\alpha$ and $\eta_c$, that is, $s>\eta_c$ holds except on the longest interval. 
Here we check the numerical results $\eta$ for different initial conditions: $\eta(x,0) = \eta_a$ and $\eta(x,0) = \eta_a + \frac{\eta_a}{2}\sin(2\pi x)$.
We let $\alpha = 1.0$ and $\eta_c = \eta_a\times 10^{-12}$ for each case.
The numerical results are plotted in Figure \ref{Fig:ex1-1} and \ref{Fig:ex1-2}. 
Let $t_1<t_2<\cdots$ denote the sequence of rupture time as in Lemma \ref{Lsing}.
 For each figure, the left side shows the profile of the numerical solution one time step before $t_1<\cdots<t_{11}$, and the right side is one at $t_1<\cdots<t_{11}$.
 These figures show the convergence of the profile of $\eta$ at rupture time.
\begin{figure}[tb]
\begin{minipage}{.48\linewidth}
\centering
\includegraphics[scale=1.0]{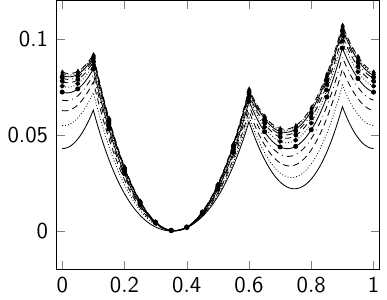}
\subcaption{$\eta$ before each rupture time}
\end{minipage}
\begin{minipage}{.48\linewidth}
\centering
\includegraphics[scale=1.0]{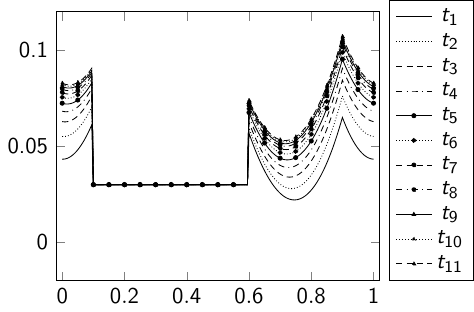}
\subcaption{$\eta$ at each rupture time}
\end{minipage}
\caption{Example \ref{Ex1}: $\eta(x,0) = \eta_a$}\label{Fig:ex1-1}
\end{figure}
\begin{figure}[tb]
\begin{minipage}{.48\linewidth}
\centering
\includegraphics[scale=1.0]{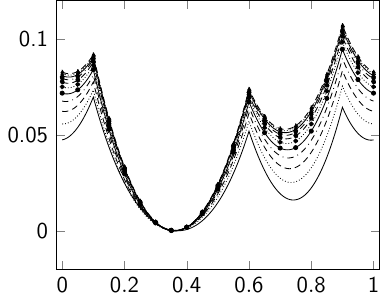}
\subcaption{$\eta$ before each rupture time}
\end{minipage}
\begin{minipage}{.48\linewidth}
\centering
\includegraphics[scale=1.0]{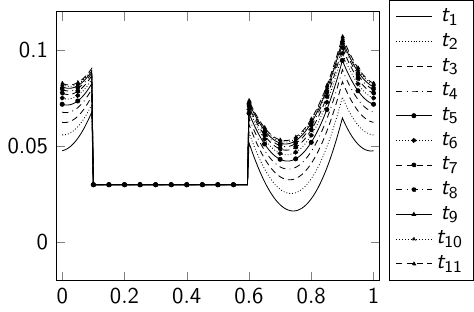}
\subcaption{$\eta$ at each rupture time}
\end{minipage}
\caption{Example \ref{Ex1}: $\eta(x,0) = \eta_a+\frac{\eta_a}{2}\sin(2\pi x)$}\label{Fig:ex1-2}
\end{figure}
In our notation, our numerical experiments indicate that $\mathcal{T}^m\eta_0$ converges to unique fixed point of $\mathcal{T}$ as $m\to\infty$ and the convergence looks monotone increasing.
\end{exam}
\begin{exam}\label{Ex2}
In this example, we also consider the case $\sigma_2=\sigma_1/\tau=1$ and $\{a_k\}_{k=1}^K = \{0.1, 0.6, 0.9\}$. 
In contrast to the previous example, we suppose that $\alpha$ and $\eta_c$ are large, then it is expected that the assumption (S) is not satisfied. Let $\alpha = 60$ and $\eta_c = \eta_a\times 10^{-1}$. 
Figure \ref{Fig:ex2} shows the numerical result $\eta$ just before each rupture time. We can find that the rupture set at each rupture time is no longer included in single interval, that is, the assumption of Lemma \ref{Lsing} does not hold. This example demonstrates the case that the assumption (S) is not satisfied, and the periodic solution may not exist. 
\end{exam}
\begin{figure}[tb]
\centering
\includegraphics[scale=1.0]{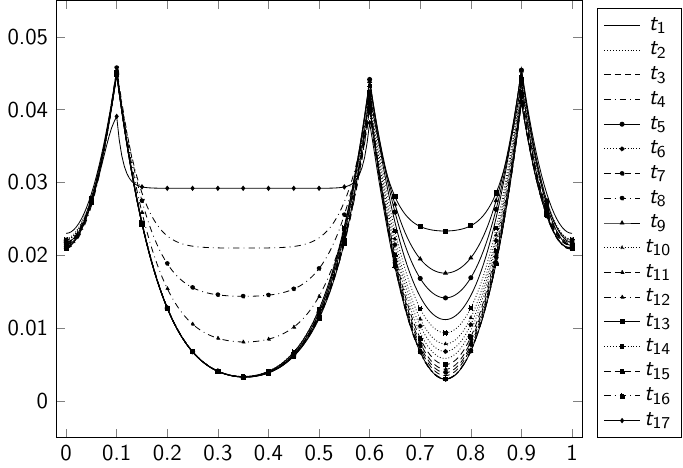}
\caption{Example \ref{Ex2}: $\eta$ before each rupture time}\label{Fig:ex2}
\end{figure}
\begin{exam}\label{Ex3}
The third example shows the numerical solution when $\sigma_1/\tau = 0.5$, $\sigma_2=1.0$, $\alpha = 1.0$ and $\eta_c = \eta_a\times 10^{-12}$. Since $\sigma_1/\tau \neq \sigma_2$, equation \eqref{Eeta} does not hold.
We have the numerical results $h$ and $\zeta$ by applying the numerical method for equation \eqref{Eh1} and \eqref{Eze}, respectively. 
Then $\eta = \zeta-h$ can be plotted for each time step. 
Figure \ref{Fig:ex3-1}, \ref{Fig:ex3-2} and \ref{Fig:ex3-3} show the graph of numerical solutions $h$, $\zeta$ and $\eta$ before the rupture time, respectively. 
The periodic behavior is not observed.
\begin{figure}[tb]
\centering
\includegraphics[scale=1.2]{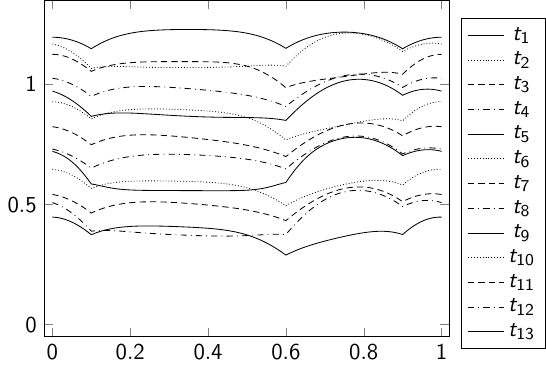}
\caption{Example \ref{Ex3}: graph of $h$ before each rupture time}\label{Fig:ex3-1}
\end{figure}
\begin{figure}[tb]
\centering
\includegraphics[scale=1.2]{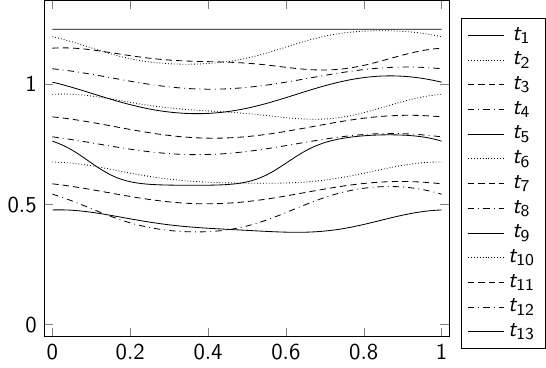}
\caption{Example \ref{Ex3}: graph of $\zeta$ before each rupture time}\label{Fig:ex3-2}
\end{figure}
\begin{figure}[tb]
\centering
\includegraphics[scale=1.2]{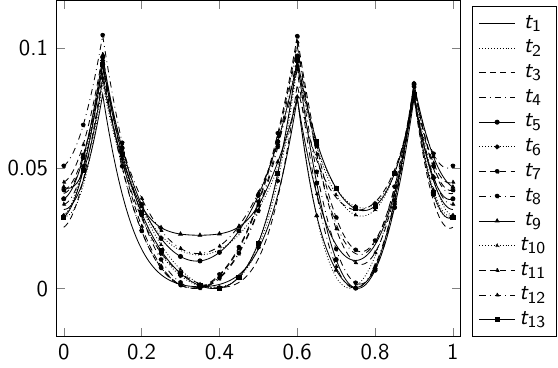}
\caption{Example \ref{Ex3}: graph of $\eta$ before each rupture time}\label{Fig:ex3-3}
\end{figure}
\end{exam}

\vspace{0.5cm}

\noindent
{\bf Acknowledgments.}\
The authors are grateful to Professor Elliott Ginder for informing them of \cite{SS} with valuable discussion.
 This work was done as a part of research activities of Social Cooperation Program ``Mathematical Sciences for Refrigerant Thermal Fluids'' sponsored by Daikin Industries, Ltd.\ at the University of Tokyo. The authors are grateful to members of the Technology and Innovation Center of Daikin Industries, Ltd.\ for showing several interesting phenomena related to rupture with fruitful discussion which triggered this work.
 The work of the first author was partly supported by the Japan Society for the Promotion of Science (JSPS) through the grants Kakenhi: No.~20K20342, No.~19H00639, and by Arithmer Inc., Daikin Industries, Ltd.\ and Ebara Corporation through collaborative grants.



\begin{thebibliography}{99}
%
\bibitem{EO}
S.~Esedo\={g}lu and F.~Otto, 
Threshold dynamics for networks with arbitrary surface tensions. 
\emph{Comm.\ Pure Appl.\ Math.}\ 68 (2015), no.~5, 808--864.
%
\bibitem{GIK}
Y.~Giga, N.~Ishimura and Y.~Kohsaka, 
Spiral solutions for a weakly anisotropic curvature flow equation.
\emph{Adv.\ Math.\ Sci.\ Appl.}\ 12 (2002), no.~1, 393--408.
%
\bibitem{GT}
D.~Gilbarg, N.~S.~Trudinger, 
Elliptic partial differential equations of second order. Reprint of the 1998 edition. Classics in Mathematics. 
\emph{Springer-Verlag, Berlin,} 2001. x\hspace{-0.1em}i\hspace{-0.1em}v+517 pp.
%
\bibitem{MNT}
C.~Mantegazza, M.~Novaga and V.~M.~Tortorelli, 
Motion by curvature of planar networks. 
\emph{Ann.\ Sc.\ Norm.\ Super.\ Pisa Cl.\ Sci.}\ (5) 3 (2004), no.~2, 235--324.
%
\bibitem{OHM}
T.~Ogiwara, D.~Hilhorst and H.~Matano, 
Convergence and structure theorems for order-preserving dynamical systems with mass conservation. 
\emph{Discrete Contin.\ Dyn.\ Syst.}\ 40 (2020), no.~6, 3883--3907.
%

\bibitem{OM}
T.~Ogiwara and H.~Matano, 
Monotonicity and convergence results in order-preserving systems in the presence of symmetry. 
\emph{Discrete Contin.\ Dynam.\ Systems} 5 (1999), no.~1, 1--34.
%
\bibitem{ON}
T.~Ogiwara and K.-I.~Nakamura, 
Spiral traveling wave solutions of some parabolic equations on annuli. 
Josai Math.\ Monogr., 2, 
\emph{Josai University, Graduate School of Science, Sakado,} 2000, 15--34.
%
\bibitem{SS}
R.~I.~Saye and J.~A.~Sethian,
Multiscale modeling of membrane rearrangement, drainage, and rupture in evolving foams. 
\emph{Science} 340 (2013), no.~6133, 720--724.
\end{thebibliography}
\end{document}